\documentclass[12pt,leqno]{amsart}
\usepackage{amssymb,amsthm,amsmath,latexsym}
\newtheorem{theorem}{\sc Theorem}
\newtheorem{lemma}[theorem]{\sc Lemma}

\newtheorem{corollary}[theorem]{\sc Corollary}

\date{}

\title{On rational and concise words}
\author{Robert Guralnick}
\address{Department of Mathematics, University of Southern California,
Los Angeles, CA 90089-1113, USA}
\email{guralnic@math.usc.edu}
\author{Pavel Shumyatsky }
\address{ Department of Mathematics, University of Brasilia,
Brasilia-DF, 70910-900 Brazil }
\email{pavel@unb.br}
\thanks{The first author was partially supported by the NSF grant DMS-1302886; the second author
was supported by CNPq-Brazil}

\begin{document}
\begin{abstract} A group-word $w$ is called concise if whenever the set of $w$-values in a group $G$ is finite it always follows that the verbal subgroup $w(G)$ is finite. More generally, a word $w$ is said to be concise in a class of groups $X$ if whenever the set of $w$-values  is finite for a group $G\in X$, it always follows that $w(G)$ is finite. P. Hall asked whether every word is concise. Due to Ivanov the answer to this problem is known to be negative. It is still an open problem whether every word is concise in the class of residually finite groups. A word $w$ is rational if the number of solutions
to the equation $w(x_1,\dots,x_k)=g$ is the same as the number of solutions to $w(x_1,\dots,x_k)=g^e$ for every finite group $G$ and for every $e$ relatively prime to $|G|$. We observe that any rational word is concise in the class of residually finite groups. Further we give a sufficient condition for rationality of a word. As a corollary we deduce that the word $w=[\dots[x_1^{n_1},x_2]^{n_2},\dots,x_k]^{n_k}$ is concise in the class of residually finite groups.
\end{abstract}
\maketitle

Let $w=w(x_1,\dots,x_k)$ be a group-word, and let $G$ be a group. The verbal subgroup $w(G)$ of $G$ determined by $w$ is the subgroup generated by the set $G_w$ consisting of all values $w(g_1,\ldots,g_k)$, where $g_1,\ldots,g_k$ are elements of $G$.  A word $w$ is said to be concise if whenever $G_w$ is finite for a group $G$, it always follows that $w(G)$ is finite. More generally, a word $w$ is said to be concise in a class of groups $X$ if whenever $G_w$ is finite for a group $G\in X$, it always follows that $w(G)$ is finite. P. Hall asked whether every word is concise, but  later Ivanov proved that this problem has a negative solution in its general form \cite{ivanov} (see also \cite[p.\ 439]{ols}). On the other hand, many relevant words are known to be concise. For instance, it was shown in \cite{jwilson} that the multilinear commutator words are concise. Such words are also known under the name of outer commutator words and are precisely the words that can be written in the form of multilinear Lie monomials. Merzlyakov showed that every word is concise in the class of linear groups \cite{merzlyakov} while Turner-Smith proved that every word is concise in the class of  residually finite groups all of whose quotients are again residually finite \cite{TS2}. There is an open problem whether every word is concise in the class of residually finite groups (cf. Segal \cite[p.\ 15]{Se} or A. Jaikin-Zapirain \cite{jaikin}). It was shown in \cite{as} that if $w$ is a multilinear commutator word and  $n$ is a prime-power, then the word $w^n$ is concise in the class of residually finite groups.

We say that a word $w$ is boundedly concise in a class of groups $X$ if for every integer $m$ there exists a number $\nu=\nu(X,w,m)$ such that whenever $|G_w|\leq m$ for a group $G\in X$ it always follows that $|w(G)|\leq\nu$. Fern\'andez-Alcober and Morigi \cite{fernandez-morigi} showed that every word which is concise in the class of all groups is actually boundedly concise. Moreover they showed that whenever $w$ is a multilinear commutator word having at most $m$ values in a group $G$, one has $|w(G)|\leq (m-1)^{(m-1)}$.
It was shown in \cite{as} that if $w=\gamma_k$ is the $k$th lower central word and $n$ a prime-power, then the word $w^n$ is boundedly concise in the class of residually finite groups. Recall that the word $\gamma_{k}$ is defined inductively by the formulae
\[
\gamma_1=x_1,
\qquad
\gamma_k=[\gamma_{k-1},x_k]=[x_1,\ldots,x_k],
\quad
\text{for $k\ge 2$.}
\]
The corresponding verbal subgroup $\gamma_k(G)$ is the familiar $k$th term of the lower central series of $G$.  

The present article grew out of the observation that if $w$ is a weakly rational word and $G$ is a residually finite group in which $w$ has at most $m$ values, then the order of $w(G)$ is $m$-bounded.

We say that $w$ is {\it weakly rational} if $g\in G_w$ if and only if $g^e \in G_w$ for every finite group $G$ and for every $e$ relatively prime to $|G|$.  We say that $w$ is {\it rational} if the number of solutions
to the equation $w(x_1,\dots,x_k)=g$ is the same as the number of solutions to $w(x_1,\dots,x_k)=g^e$.  Clearly rational implies weakly rational.

\begin{lemma}\label{e1}
The word $w$ is weakly rational if and only if $g^e\in G_w$ for every finite group $G$, every $g \in G_w$, and every $e$ relatively prime to $|g|$. The word $w$ is rational if and only if the number of solutions
to the equation $w(x_1,\dots,x_k)=g$ is the same as that to $w(x_1,\dots,x_k)=g^e$.
\end{lemma}
\begin{proof} We will prove only the claim about weakly rational words as the other statement follows by a similar argument. It is clear that if $g^e \in G_w$ for every finite group $G$, every $g \in G_w$, and every $e$ relatively prime to $|g|$, then $w$ is weakly rational. Let us prove the converse. Assume that $w$ is weakly rational. Let $G$ be a finite group and $g \in G_w$. Given $e$ such that $(e,|g|)=1$, we need to show that $g^e\in G_w$.

Let $d$ be the maximal divisor of $|G|$ such that $(d,e)=1$. Set $e_1=e+d$. Since $|g|$ divides $d$, it follows that $g^e=g^{e_1}$. On the other hand, it is clear that $e_1$ is relatively prime to $|G|$. Hence, $g^e\in G_w$.
\end{proof}

\begin{lemma}\label{concise} Let $m$ be a positive integer and $w$ a weakly rational word. Let $G$ be a residually finite group with at most $m$ values of the word $w$  Then the order of $w(G)$ is $m$-bounded.
\end{lemma}
\begin{proof} We can assume that $G$ is finite. Let $W=w(G)$. The group $G$ acts on the set of $w$-values by conjugation and therefore $G/C_G(W)$ embeds in the symmetric group on $m$ symbols. It follows that the order of $W/Z(W)$ is at most $m!$. By Schur's theorem $W'$ has $m$-bounded order \cite[10.1.4]{rob}. Passing to $G/W'$ we can assume that $W$ is abelian. Let $g$ be a $w$-value. Lemma \ref{e1} shows that every element $h$ such that $\langle h\rangle=\langle g\rangle$ is a $w$-value. The number of $h$ with this property is $\phi(|g|)$, where $\phi$ denotes the Euler function. Thus, $\phi(|g|)\leq m$. Of course, this implies that $|g|$ is $m$-bounded. Hence, $W$ is an abelian group generated by $m$ elements each of which has $m$-bounded order. We conclude that the order of $W$ is $m$-bounded.
\end{proof}

Many words are rational. For example, $x^n$ is obviously rational for any positive integer $n$. More generally, it is easy to see that $w^n$ is (weakly) rational whenever $w$ is. It is well-known that the commutator word $w=[x,y]$ is rational (cf \cite{Honda, Isaacs}). On the other hand, by a result of Lubotzky \cite{Lu}, there are many words that are not (weakly) rational.   The main result of Lubotzky shows that
if $G$ is a finite simple group and $X$ is a subset of $G$ closed under automorphisms and containing
$1$, then $X$ is the image of some word $w$ on $G$ (in two variables).   In particular,  if take 
$G=\mathrm{PSL}_2(p)$ with $p$ a prime at least $7$,  there is a word $w$ so that $w(G)$ consists
of the conjugacy class  $C$ of elements of order $(p+1)/2$ together with $1$.  Note that 
if  $x \in C$ and $\langle x \rangle = \langle y \rangle$, then $y \in C$
if and only if $x=y$ or
$x=y^{-1}$.  In particular, the image of $w$ is not rational (since there are generators for 
$\langle x \rangle, x \in C$ other than $x$ and $x^{-1}$).  In particular, $w$ is not weakly rational.
One can construct similar examples for almost all simple groups of Lie type. 

Lubotzky's result also shows that the image of $w$ can be small compared to $|w(G)|$.
For example, we can take $G=\mathrm{Alt}_n$ with $w(G)$ being the set of $3$-cycles together
with $1$.  Such words were also described in \cite{KN}.

Further examples of rational words can be obtained via the following theorem.

\begin{theorem}\label{ra}
Suppose that $w = w(x_1, \ldots, x_r)$ is a word in $r$-variables.  Then $w$
weakly rational or rational implies the same for
$w':=[w,x_0]$.
\end{theorem}

\begin{proof} Let $C$ be a conjugacy class in $G_w$.  If $e$ is prime to $|G|$ and 
$w$ is weakly rational, then $C^e$ is also contained in $G_w$.  
If $w$ is rational, then the number of solutions of $w=g$ is the same as for $g^e$.

Fix $g\in G$ and let $D=g^G$.  Let us count the number of 
triples $N:=N(D,C,C^{-1})$ of $(a,b,c) \in D^{-1} \times C^{-1} \times C$ such that  $abc=1$.  
It is well known that:

$$
N=\frac{|C|^2|D|}{|G|} \sum_{\chi}  \frac{\chi(g^{-1}) \chi(b) \chi(b^{-1})}{\chi(1)},
$$
where the sum is over all irreducible characters $\chi$ and $b$ is some element of $C^{-1}$.

Now apply the Galois automorphism of $\mathbb{Q}[\theta]$ sending $\theta$ to $\theta^e$
for $e$ prime to $|G|$ with $\theta$ a primitive $|G|$th root of $1$.  This shows that
$N=N(D^e,C^e,C^{-e})$.   

In particular,  if $g$ is in the image of $w'$, then $g=[y,x]$ for some 
$y \in C \subset G_w$ and $x\in G$.
This is equivalent to saying that $N(D,C,C^{-1}) \ne 0$,
 whence $N(D^e,C^e,C^{-e}) \ne 0$ and so if $w$ weakly rational, $g^e$ 
is in the image of $w'$.   If $w$ is rational, then we look at 
the equality $N(D^e,C^e,C^{-e})=N(D,C,C^{-1})$ over each conjugacy 
class contained in $G_w$ and conclude that $w'$ is also rational.
\end{proof}

\begin{corollary} Let $n_1,\dots,n_k$ be positive integers and 
$$w=[\dots[x_1^{n_1},x_2]^{n_2},\dots,x_k]^{n_k}.$$ 
If $G$ is a residually finite group in which $w$ has at most $m$ values, 
then the order of $w(G)$ is $m$-bounded.
\end{corollary}

\begin{proof} In view of Lemma \ref{concise} it is sufficient to show that the 
word $w$ is weakly rational. Taking this into account that whenever a word $v$ 
is (weakly) rational so is the word $v^n$ and using Theorem \ref{ra}, 
straightforward induction on $k$ shows that the word $w$ is actually rational.
\end{proof}

\end{document}